\documentclass[12pt]{article}%
\usepackage{amsmath}
\usepackage{amsfonts}
\usepackage{amssymb}
\usepackage{graphicx}%
\makeatletter
\newfont{\footsc}{cmcsc10 at 8truept}
\newfont{\footbf}{cmbx10 at 8truept}
\newfont{\footrm}{cmr10 at 10truept}
\newtheorem{theorem}{Theorem}

\newtheorem{conjecture}[theorem]{Conjecture}

\newtheorem{proposition}[theorem]{Proposition}

\newenvironment{proof}[1][Proof]{\noindent{\textbf {#1}  }}  {\hfill$\blacksquare$\bigskip}

\def\blfootnote{\xdef\@thefnmark{}\@footnotetext}
\begin{document}

\title{Maxima of the $Q$-index: forbidden $4$-cycle and $5$-cycle}
\author{Maria Aguieiras A. de Freitas\thanks{Federal University of Rio de Janeiro,
Brazil; \textit{email: maguieiras@im.ufrj.br}} , Vladimir
Nikiforov\thanks{Department of Mathematical Sciences, University of Memphis,
Memphis TN 38152, USA; \textit{email: vnikiforv@memphis.edu}} \ and Laura
Patuzzi\thanks{Federal University of Rio de Janeiro, Brazil; \textit{email:}
\textit{laura@im.ufrj.br}}}
\date{}
\maketitle

\begin{abstract}
This paper gives tight upper bounds on the largest eigenvalue $q\left(
G\right)  $ of the signless Laplacian of graphs with no $4$-cycle and no $5$-cycle.

If $n$ is odd, let $F_{n}$ be the friendship graph of order $n;$ if $n$ is
even, let $F_{n}$ be $F_{n-1}$ with an edge hanged to its center. It is shown
that if $G$ is a graph of order $n\geq4,$ with no $4$-cycle, then
\[
q\left(  G\right)  <q\left(  F_{n}\right)  ,
\]
unless $G=F_{n}.$

Let $S_{n,k}$ be the join of a complete graph of order $k$ and an independent
set of order $n-k.$ It is shown that if $G$ is a graph of order $n\geq6,$ with
no $5$-cycle, then
\[
q\left(  G\right)  <q\left(  S_{n,2}\right)  ,
\]
unless $G=S_{n,k}.$

It is shown that these results are significant in spectral extremal graph
problems. Two conjectures are formulated for the maximum $q\left(  G\right)  $
of graphs with forbidden cycles.\medskip



\end{abstract}

\section{Introduction}

Given a graph $G,$ the $Q$-index of $G$ is the largest eigenvalue $q\left(
G\right)  $ of its signless Laplacian $Q\left(  G\right)  $. In this note we
determine the maximum $Q$-index of graphs with no cycles of length $4$ and
$5.$ These two extremal problems turn out to have particular meaning, so
before getting to the new theorems, we propose an introductory discussion.

Recall that the central problem of the classical extremal graph theory is of
the following type: \medskip

\textbf{Problem A }\emph{Given a graph }$F,$\emph{ what is the maximum number
of edges of a graph of order }$n,$\emph{ with no subgraph isomorphic to
}$F?\medskip$

Such problems are fairly well understood nowadays; see, e.g., \cite{Bol78} for
comprehensive discussion and \cite{Nik11} for some newer results. Let us
mention only two concrete results, which will be used for case study later in
the paper. Write $T_{2}\left(  n\right)  $ for the complete bipartite graph
with parts of size $\left\lfloor n/2\right\rfloor $ and $\left\lceil
n/2\right\rceil ,$ and note that $T_{2}\left(  n\right)  $ contains no odd
cycles. First, Mantel's theorem \cite{Man07} reads as: \emph{if }$G$\emph{ is
a graph of order }$n$ and \emph{ }$e\left(  G\right)  >e\left(  T_{2}\left(
n\right)  \right)  ,$ \emph{then }$G$ \emph{contains a triangle}$.$ This
theorem seems as good as one can get, but Bollob\'{a}s in \cite{Bol78}, p.
150, was able to deduce a stronger conclusion from the same premise: \emph{if
}$G$\emph{ is a graph of order }$n$\emph{ and }$e\left(  G\right)  >e\left(
T_{2}\left(  n\right)  \right)  $\emph{, then }$G$\emph{ contains a cycle of
length }$t$\emph{ for every }$3\leq t\leq\left\lceil n/2\right\rceil .$ Below
we study similar problems for spectral parameters of graphs.

During the past two decades, subtler versions of Problem A have been
investigated, namely for the spectral radius $\mu\left(  G\right)  $, i.e.,
the largest eigenvalue of the adjacency matrix of a graph $G.$ In this new
class of problems the central question is the following one.\medskip

\textbf{Problem B }\emph{Given a graph }$F,$\emph{ what is the maximum }%
$\mu\left(  G\right)  $\emph{ of a graph }$G$ \emph{of order }$n,$\emph{ with
no subgraph isomorphic to }$F.\medskip$

In fact, in recent years much of the classical extremal graph theory, as known
from \cite{Bol78}, has been recast for the spectral radius with an astonishing
preservation of detail; see \cite{Nik11} for a survey and discussions. In
particular, Mantel's theorem has been translated as: \emph{if }$G$\emph{ is a
graph of order }$n$ \emph{and }$\mu\left(  G\right)  >\mu\left(  T_{2}\left(
n\right)  \right)  $\emph{, then }$G$ \emph{contains a triangle}$.$\emph{
}Moreover, as shown in \cite{Nik08}, the result of Bollob\'{a}s can be
recovered likewise: \emph{if }$G$\emph{ is a graph of sufficiently large order
}$n$ \emph{and }$\mu\left(  G\right)  >\mu\left(  T_{2}\left(  n\right)
\right)  ,$\emph{ then }$G$\emph{ contains a cycle of length }$t$\emph{ for
every }$3\leq t\leq n/320.$ The constant $1/320$ undoubtedly can be improved,
but its best value is not known yet.

However, not all extremal results about the spectral radius are similar to the
corresponding edge extremal results. The most notable known exceptions are for
the spectral radius of graphs with forbidden even cycles. Thus, it has been
long known (\cite{KST54},\cite{Rei58}) that the number of edges in a graph $G$
of order $n,$ with no $4$-cycle, satisfies
\[
e\left(  G\right)  \leq\frac{1}{4}n\left(  1+\sqrt{4n-3}\right)  ,
\]
and F\"{u}redi \cite{Fur88} proved that, for $n$ sufficiently large, equality
holds if and only if $G$ is a the polarity graph of Erd\H{o}s and Renyi
\cite{ErRe62}. By contrast, it has been shown in \cite{Nik07}, \cite{Nik09}
and \cite{ZhWa12}, that the spectral radius of a graph $G$ of order $n,$ with
no $4$-cycle, satisfies
\[
\mu\left(  G\right)  \leq\mu\left(  F_{n}\right)  ,
\]
where for $n$ odd, $F_{n}$ is a union of $\left\lfloor n/2\right\rfloor $
triangles sharing a single common vertex, and for $n$ even, $F_{n\text{ }}$ is
obtained by hanging an edge to the common vertex of $F_{n-1}.$ Note that for
odd $n$ the graph $F_{n}$ is known also as the \emph{friendship} or
\emph{windmill} graph.

For longer even cycles the results for edges and the spectral radius deviate
even further. Indeed, if $G$ is a graph of order $n,$ with no cycle of length
$2k,$ finding the maximum number of edges in $G$ is one of the most difficult
problems in extremal graph theory, far from solution for any $k\geq4$. By
contrast, the maximum spectral radius is known within $1/2$ of its best
possible value, although the problem is not solved completely - see Conjecture
\ref{con} at the end.

The present paper contributes to an even newer trend in extremal graph theory,
namely to the study of variations of Problem A for the $Q$-index of graphs,
where the central question is the following one.\medskip\medskip

\textbf{Problem C }\emph{Given a graph }$F,$\emph{ what is the maximum }%
$Q$\emph{-index a graph }$G$ \emph{of order }$n,$\emph{ with no subgraph
isomorphic to }$F?\medskip$

This question has been resolved for forbidden $K_{r},$ $r\geq3$ in
\cite{AbNi13} and \cite{HJZ13}, obtaining stronger versions of the Tur\'{a}n
theorem. In general, extremal problems about the $Q$-index turn out to be more
difficult than for the spectral radius and show greater deviation from their
edge versions. For example, Mantel's theorem for the $Q$-index reads as:
\emph{if }$G$\emph{ is a graph of order }$n$ and $q\left(  G\right)  >q\left(
T_{2}\left(  n\right)  \right)  ,$ \emph{then }$G$ \emph{contains a
triangle}$.$ While this result is expected, here $T_{2}\left(  n\right)  $ is
not the only extremal graph, unlike the case for $e\left(  G\right)  $ and
$\mu\left(  G\right)  $. An even greater rift is observed when we try to
translate the result of Bollob\'{a}s: now the premise $q\left(  G\right)
>q\left(  T_{2}\left(  n\right)  \right)  $ by no means implies the existence
of cycles other than triangles. Indeed, write $S_{n,1}^{+}$ for the star of
order $n$ with an additional edge. Clearly $S_{n,1}^{+}$ contains no cycles
other than triangles, although $q\left(  S_{n,1}^{+}\right)  >q\left(
T_{2}\left(  n\right)  \right)  .$ In fact, to guarantee existence of an odd
cycle of length $k,$ we need considerably larger $q\left(  G\right)  .$ Below,
we find the precise lower bound for $5$-cycles, and we state a conjecture for
general $k$ in the concluding remarks.

For even cycles, $q\left(  G\right)  $ and $\mu\left(  G\right)  $ seem to
behave alike, although we are able to prove this fact only for the $4$-cycle;
we state a general conjecture at the end of the paper.

\begin{theorem}
\label{thc4}Let $G$ be a graph of order $n\geq4.$ If $G$ contains no $C_{4}$,
then
\[
q\left(  G\right)  <q\left(  F_{n}\right)  ,
\]
unless $G=F_{n}.$
\end{theorem}

The following proposition adds some numerical estimates to Theorem \ref{thc4}.

\begin{proposition}
\label{pro1}If $n$ is odd, then
\[
q\left(  F_{n}\right)  =\frac{n+2+\sqrt{\left(  n-2\right)  ^{2}+8}}{2},
\]
and satisfies
\begin{equation}
n+\frac{2}{n-1}<q\left(  F_{n}\right)  <n+\frac{2}{n-2}. \label{bodd}%
\end{equation}

If $n$ is even, then $q\left(  F_{n}\right)  $ is the largest root of the
equation
\[
x^{3}-\left(  n+3\right)  x^{2}+3nx-2n+4=0,
\]
and satisfies%
\begin{equation}
n+\frac{2}{n}<q\left(  F_{n}\right)  <n+\frac{2}{n-1}. \label{beven}%
\end{equation}

\end{proposition}

Let now $S_{n,k}$ be the graph obtained by joining each vertex of $K_{k},$ the
complete graph of order $k,$ to each vertex of an independent set of order
$n-k;$ in other words, $S_{n,k}=K_{k}\vee\overline{K}_{n-k}.$

\begin{theorem}
\label{thc5}Let $n\geq6$ and let $G\ $be a graph of order $n.$ If $G$ contains
no $C_{5}$, then
\[
q\left(  G\right)  <q\left(  S_{n,2}\right)  ,
\]
unless $G=S_{n,2}.$
\end{theorem}

Note that for $n=5,$ there are two graphs without $C_{5}$ and with maximal
$Q$-index: one is $S_{5,2},$ and the other is $K_{4}$ with a dangling
edge.\medskip

The remaining part of the paper is organized as follows. In the next section
we give the proofs of Theorems \ref{thc4}, \ref{thc5} and Proposition
\ref{pro1}. In the concluding remarks we round up the general discussion and
state two conjectures.

\section{Proofs}

For graph notation and concepts undefined here, we refer the reader to
\cite{Bol98}. For introductory material on the signless Laplacian see the
survey of Cvetkovi\'{c} \cite{C10} and its references. In particular, let $G$
be a graph, and $X$ and $Y$ be disjoint sets of vertices of $G.$ We write:

- $V\left(  G\right)  $ for the set of vertices of $G,$ and $e\left(
G\right)  $ for the number of its edges;

- $G\left[  X\right]  $ for the graph induced by $X,$ and $e\left(  X\right)
$ for $e\left(  G\left[  X\right]  \right)  ;$

- $e\left(  X,Y\right)  $ for the number of edges joining vertices in $X$ to
vertices in $Y;$

- $\Gamma\left(  u\right)  $ for the set of neighbors of a vertex $u,$ and
$d\left(  u\right)  $ for $\left\vert \Gamma\left(  u\right)  \right\vert
.\medskip$

Write $P_{k}$ for a path of order $k,$ and recall that Erd\H{o}s and Gallai
\cite{ErGa59} have shown that if $G$ is a graph of order $n$ with no
$P_{k+2},$ then $e\left(  G\right)  \leq kn/2,$ with equality holding if and
only if $G$ is a union of complete graphs of order $k+1.$\bigskip

\begin{proof}
[\textbf{Proof of Proposition \ref{pro1}}]If n is odd, then by definition
\[
F_{n}=K_{1}\vee\left(  \frac{n-1}{2}\right)  K_{2}.
\]
Then the $Q$-index of $F_{n}$ (see \cite{MNRS}) is the largest root of the
equation
\[
x^{2}-\left(  n+2\right)  x+2\left(  n-1\right)  =0,
\]
that is,
\[
q\left(  F_{n}\right)  =\frac{n+2+\sqrt{\left(  n-2\right)  ^{2}+8}}{2}.
\]
An easy check shows that%
\[
\left(  n-2+\frac{4}{n-1}\right)  ^{2}<(n-2)^{2}+8<\left(  n-2+\frac{4}%
{n-2}\right)  ^{2},
\]
proving (\ref{bodd}).

If $n$ is even, then by definition
\[
F_{n}=K_{1}\vee\left(  K_{1}\cup\left(  \frac{n-2}{2}\right)  K_{2}\right)  .
\]
Then, its $Q$-index (see \cite{MNRS}) is the largest root of the polynomial
\begin{align*}
p(x) &  =x^{3}-(n+3)x^{2}+3nx-2n+4\\
&  =(x-n)^{3}+(2n-3)(x-n)^{2}+(n^{2}-3n)(x-n)-2n+4.
\end{align*}
Since $p(x)$ is an increasing function for $x\geq n$, we find that
\[
p\left(  n+\frac{2}{n}\right)  =\frac{-2}{n^{3}}\left(  n^{2}%
(n-4)+2(3n-2)\right)  <0
\]
and so,%
\[
q(F_{n})>n+\frac{2}{n}.
\]
On the other hand
\[
p\left(  n+\frac{2}{n-1}\right)  =\frac{4}{(n-1)^{2}}\left(  n-2+\frac{2}%
{n-1}\right)  >0,
\]
and it follows that
\[
q(F_{n})<n+\frac{2}{n-1},
\]
completing the proof of (\ref{beven}).
\end{proof}

\bigskip

The proofs of Theorem \ref{thc4} and Theorem \ref{thc5} will be based on a
careful analysis of the following bound on $q\left(  G\right)  ,$ which can be
traced back to Merris \cite{Mer98},
\begin{equation}
q\left(  G\right)  \leq\max_{u\in V\left(  G\right)  }d\left(  u\right)
+\frac{1}{d\left(  u\right)  }\sum_{v\in\Gamma\left(  u\right)  }d\left(
v\right)  \label{Min}%
\end{equation}

Let us also note that for every vertex $u\in V\left(  G\right)  ,$
\begin{equation}
\sum_{v\in\Gamma\left(  u\right)  }d\left(  v\right)  =2e\left(  \Gamma\left(
u\right)  \right)  +e\left(  \Gamma\left(  u\right)  ,V\backslash\Gamma\left(
u\right)  \right)  \label{aux}%
\end{equation}
\bigskip

\begin{proof}
[\textbf{Proof of Theorem \ref{thc4}}]Suppose that $G$ is a graph of order
$n\geq4,$ with no $C_{4}$, and let $\Delta$ be the maximum degree of $G.$ Note
first that if $\Delta=n-1$, then $G$ is a connected subgraph of $F_{n},$
hence, $q\left(  G\right)  <q\left(  F_{n}\right)  ,$ unless $G=F_{n}.$ Thus,
hereafter we shall assume that $\Delta\leq n-2.$

Let $u$ be a vertex for which the maximum in the right-hand side of
(\ref{Min}) is attained, and set for short $N=\Gamma\left(  u\right)  $. If
$d\left(  u\right)  =1,$ then%
\[
q\left(  G\right)  \leq d\left(  u\right)  +\frac{1}{d\left(  u\right)  }%
\sum_{v\in N}d\left(  v\right)  \leq1+\frac{\Delta}{1}\leq n-1<q\left(
F_{n}\right)  ,
\]
proving the assertion in this case. Thus, hereafter we shall assume that
$d\left(  u\right)  \geq2.$

Note that every vertex $v\in V\backslash\left\{  u\right\}  $ has at most one
neighbor in $N$ because $G$ has no $C_{4}.$ Therefore,
\[
e\left(  N,V\backslash N\right)  =d\left(  u\right)  +\left\vert
V\backslash\left(  N\cup\left\{  u\right\}  \right)  \right\vert =d\left(
u\right)  +n-d\left(  u\right)  -1=n-1,
\]
and so,
\[
\sum_{v\in N}d\left(  v\right)  \leq d\left(  u\right)  +e\left(
N,V\backslash N\right)  \leq d\left(  u\right)  +n-1,
\]
implying that,
\[
q\left(  G\right)  \leq d\left(  u\right)  +\frac{1}{d\left(  u\right)  }%
\sum_{v\in N}d\left(  v\right)  \leq1+d\left(  u\right)  +\frac{n-1}{d\left(
u\right)  }.
\]
Since the function
\[
f\left(  x\right)  =x+\frac{n-1}{x}%
\]
is convex for $x>0,$ its maximum in any closed interval is attained at one of
the ends of this interval. In our case, $2\leq d\left(  u\right)  \leq n-2,$
and so,%
\[
q\left(  G\right)  \leq d\left(  u\right)  +\frac{1}{d\left(  u\right)  }%
\sum_{v\in N}d\left(  v\right)  \leq1+\max\left\{  2+\frac{n-1}{2}%
,n-2+\frac{n-1}{n-2}\right\}  \leq n+\frac{1}{n-2}<q\left(  F_{n}\right)  ,
\]
completing the proof of Theorem \ref{thc4}.
\end{proof}

To simplify the proof of Theorem \ref{thc5}, we shall prove two auxiliary
statements. First, note the following relations.

\begin{proposition}
\label{pro2}If $n\geq4,$ then%
\begin{equation}
q\left(  S_{n,2}\right)  =\frac{n+2+\sqrt{n^{2}+4n-12}}{2}>n+2-\frac{4}{n+1}.
\label{b1}%
\end{equation}

\end{proposition}

\begin{proof}
Since $S_{n,2}=K_{2}\vee\overline{K}_{n-2}$, its $Q$-index is the largest root
of equation
\[
x^{2}-(n+2)x+4=0,
\]
and so,
\[
q(S_{n,2})=\frac{n+2+\sqrt{n^{2}+4n-12}}{2}.
\]
Since $n\geq4$, we find that
\begin{align*}
\sqrt{n^{2}+4n-12}  &  =\sqrt{(n+2)^{2}-16}>\sqrt{(n+2)^{2}-16\left(
1+\frac{n-3}{(n+1)^{2}}\right)  }\\
&  =n+2-\frac{8}{n+1},
\end{align*}
implying (\ref{b1}).
\end{proof}

\bigskip

In the following proposition we prove a crucial case of Theorem \ref{thc5}.

\begin{proposition}
\label{pro3}Let $G$ be a graph of order $n\geq6,$ with no $C_{5}.$ If $G$ has
a vertex of degree $n-1$, then
\[
q\left(  G\right)  <q\left(  S_{n,2}\right)  ,
\]
unless $G=S_{n,2}.$
\end{proposition}

\begin{proof}
Let $u$ be a vertex of degree $n-1$ and write $N$ for the set of its
neighbors. As proved by Chang and Tam \cite{ChTa11}, if a graph $G$ has a
dominating vertex and $4\leq e\left(  N\right)  \leq n-2,$ then the maximum
$q\left(  G\right)  $ is attained uniquely for $G=S_{n,2}.$ So it remains to
consider the cases $e\left(  N\right)  \leq3$ and $e\left(  N\right)  \geq
n-1.$ Assume for a contradiction that $q\left(  S_{n,2}\right)  \leq q\left(
G\right)  $ and $G\neq S_{n,2}.$ By the inequality of Das \cite{Das04}, we
have
\[
q\left(  S_{n,2}\right)  \leq q\left(  G\right)  \leq\frac{2e\left(  G\right)
}{n-1}+n-2,
\]
and (\ref{b1}) implies that%
\[
n+2-\frac{4}{n+1}\leq\frac{2e\left(  G\right)  }{n-1}+n-2.
\]
After some rearrangement we get
\[
e\left(  N\right)  \geq n-1-2+\frac{4}{n+1}%
\]
and so $e\left(  N\right)  \geq n-2\geq4.$ Since $C_{5}\nsubseteq G,$ the
graph $G\left[  N\right]  $ contains no $P_{4},$ and so $e\left(  N\right)
\leq n-1,$ implying that $e\left(  N\right)  =n-1$. By the theorem of
Erd\H{o}s and Gallai, this is only possible if $G\left[  N\right]  $ is a
union of disjoint triangles, that is to say,%
\[
G=K_{1}\vee\left(  \frac{n-1}{3}\right)  K_{3}.
\]
Now, an easy calculation gives%
\[
q\left(  G\right)  =\frac{n+4+\sqrt{(n-4)^{2}+16}}{2}<q\left(  S_{n,2}\right)
,
\]
a contradiction completing the proof of Proposition \ref{pro3}.
\end{proof}

\bigskip

\begin{proof}
[\textbf{Proof of Theorem \ref{thc5}}]Let $G$ be a graph of order $n\geq6,$
with no $C_{5}$ and with
\begin{equation}
q\left(  G\right)  \geq q\left(  S_{n,2}\right)  \geq n+2-\frac{4}{n+1}.
\label{in1}%
\end{equation}
Assume for a contradiction that $G\neq S_{n,2},$ and so, in view of
Proposition \ref{pro3}, we can assume that $\Delta\left(  G\right)  \leq n-2.$

Let $u$ be a vertex for which the maximum in the right-hand side of
(\ref{Min}) is attained and set for short $N=\Gamma\left(  u\right)  $. Note
that (\ref{Min}) and (\ref{in1}) give%
\begin{equation}
n+2-\frac{4}{n+1}\leq d\left(  u\right)  +\frac{1}{d\left(  u\right)  }%
\sum_{v\in N}d\left(  v\right)  .\label{in2}%
\end{equation}
We shall deduce first that
\begin{equation}
e\left(  N\right)  \geq d\left(  u\right)  -1.\label{in3}%
\end{equation}
Indeed, a crude estimate gives%
\[
\frac{e\left(  N,V\backslash N\right)  }{d\left(  u\right)  }\leq
\frac{d\left(  u\right)  \left(  n-d\left(  u\right)  \right)  }{d\left(
u\right)  }\leq n-d\left(  u\right)  ,
\]
and (\ref{aux}) implies that
\[
d\left(  u\right)  +\frac{1}{d\left(  u\right)  }\sum_{v\in N}d\left(
v\right)  \leq d\left(  u\right)  +\frac{2e\left(  N\right)  }{d\left(
u\right)  }+n-d\left(  u\right)  =n+\frac{2e\left(  N\right)  }{d\left(
u\right)  }.
\]
Hence, in view of (\ref{in2}), we get
\[
e\left(  N\right)  \geq d\left(  u\right)  -\frac{2d\left(  u\right)  }%
{n+1}\geq d\left(  u\right)  -\frac{2\left(  n-1\right)  }{n+1}>d\left(
u\right)  -2,
\]
and (\ref{in3}) follows.

Consider first that case when $G\left[  N\right]  $ contains no cycles, and so
$G\left[  N\right]  $ is a tree; moreover, since $C_{5}\nsubseteq G,$ the
graph $G\left[  N\right]  $ contains no $P_{4},$ and so it is a star. Note
that to avoid $C_{5},$ every vertex in $V\backslash\left(  N\cup\left\{
u\right\}  \right)  $ may be joined to at most one vertex of $N.$ Therefore,%
\[
e\left(  N,V\backslash N\right)  \leq n-d\left(  u\right)  -1+d\left(
u\right)  =n-1,
\]
implying that%
\begin{align*}
n+2-\frac{4}{n+1}  &  \leq d\left(  u\right)  +\frac{1}{d\left(  u\right)
}\sum_{v\in N}d\left(  v\right)  \leq d\left(  u\right)  +\frac{2\left(
d\left(  u\right)  -1\right)  +n-1}{d\left(  u\right)  }\\
&  =d\left(  u\right)  +2+\frac{n-3}{d\left(  u\right)  }.
\end{align*}
Since the function
\[
f\left(  x\right)  =x+\frac{n-3}{x}%
\]
is convex for $x>0,$ its maximum in any closed interval is attained at the
ends of this interval. In our case, $1\leq d\left(  u\right)  \leq n-2,$ and
so,%
\[
d\left(  u\right)  +2+\frac{n-3}{d\left(  u\right)  }\leq2+\max\left\{
1+\frac{n-3}{1},n-2+\frac{n-3}{n-2}\right\}  \leq n+2-\frac{4}{n+1}<q\left(
S_{n,2}\right)  ,
\]
which contradicts (\ref{in2}). Hence, $G\left[  N\right]  $ contains cycles.
However, since $C_{5}\nsubseteq G,$ all cycles of $G\left[  N\right]  $ are
triangles. Let say, $x,y,z$ are the vertices of a triangle in $G\left[
N\right]  .$ Clearly, to avoid $C_{5},$ every vertex in $V\backslash\left(
N\cup\left\{  u\right\}  \right)  $ may be joined to at most one of the
vertices $x,y,z.$ This requirement reduces the bound on $e\left(
N,V\backslash N\right)  $ as follows
\[
e\left(  N,V\backslash N\right)  \leq\left(  n-d\left(  u\right)  \right)
d\left(  u\right)  -2\left(  n-d\left(  u\right)  -1\right)  .
\]
Hence, in view of (\ref{in1}),
\begin{align*}
n+2-\frac{4}{n+1}  &  \leq d\left(  u\right)  +\frac{1}{d\left(  u\right)
}\sum_{v\in N}d\left(  v\right)  =d\left(  u\right)  +\frac{2e\left(
N\right)  }{d\left(  u\right)  }+\frac{e\left(  N,V\backslash N\right)
}{d\left(  u\right)  }\\
&  \leq d\left(  u\right)  +\frac{2e\left(  N\right)  }{d\left(  u\right)
}+n-d\left(  u\right)  -\frac{2\left(  n-d\left(  u\right)  -1\right)
}{d\left(  u\right)  }\\
&  =n+2+\frac{2e\left(  N\right)  }{d\left(  u\right)  }-\frac{2\left(
n-1\right)  }{d\left(  u\right)  }.
\end{align*}
After some rearrangement we see that%
\[
e\left(  N\right)  \geq n-1-\frac{2d\left(  u\right)  }{n+1}\geq
n-1-\frac{2\left(  n-2\right)  }{n+1}=n-3+\frac{6}{n+1},
\]
which implies that
\[
e\left(  N\right)  \geq n-2.
\]
Since $G\left[  N\right]  $ contains no $P_{4},$ the Erd\H{o}s-Gallai theorem
says that
\[
d\left(  u\right)  =v\left(  G\left[  N\right]  \right)  \geq e\left(
N\right)  ,
\]
with equality holding if an only if $G\left[  N\right]  $ is a union of
disjoint triangles. Therefore, we have $d\left(  u\right)  =n-2$ and $e\left(
N\right)  =n-2,$ and so $G\left[  N\right]  $ is a union of disjoint
triangles. As above, we see that every vertex in $V\backslash\left(
N\cup\left\{  u\right\}  \right)  $ may be joined to at most one vertex of
each triangle, and so
\[
e\left(  N,V\backslash N\right)  \leq d\left(  u\right)  +\frac{d\left(
u\right)  }{3}<\frac{4\left(  n-2\right)  }{3}.
\]
Hence,%
\begin{align*}
\frac{n+2+\sqrt{n^{2}+4n-12}}{2}  &  \leq d\left(  u\right)  +\frac
{1}{d\left(  u\right)  }\sum_{v\in N}d\left(  v\right)  \leq d\left(
u\right)  +\frac{2e\left(  N\right)  }{d\left(  u\right)  }+\frac{e\left(
N,V\backslash N\right)  }{d\left(  u\right)  }\\
&  \leq n-2+\frac{2\left(  n-2\right)  +\frac{4\left(  n-2\right)  }{3}}%
{n-2}=n+\frac{4}{3}%
\end{align*}
which is a contradiction for $n\geq5.$ Theorem \ref{thc5} is proved.
\end{proof}

\section{Concluding remarks}

We have started the study of the maximum $Q$-index of graphs with forbidden
cycles. These questions are very different from the corresponding problems
about edges and spectral radius. The main peculiarity for the $Q$-index is
that it seems to behave similarly for odd and even forbidden cycles. Here is a
general conjecture for odd cycles.

\begin{conjecture}
Let $k\geq2$ and let $G$ be a graph of sufficiently large order $n.$ If $G$
has no $C_{2k+1},$ then
\[
q\left(  G\right)  <q\left(  S_{n,k}\right)  ,
\]
unless $G=S_{n,k}.$
\end{conjecture}

To state the conjecture for even cycles, define $S_{n,k}^{+}$ as the graph
obtained by adding an edge to $S_{n,k}.$

\begin{conjecture}
Let $k\geq2$ and let $G$ be a graph of sufficiently large order $n.$ If $G$
has no $C_{2k+2},$ then
\[
q\left(  G\right)  <q\left(  S_{n,k}^{+}\right)  ,
\]
unless $G=S_{n,k}^{+}.$
\end{conjecture}

For comparison we would like to reiterate the corresponding conjecture for the
spectral radius, stated in \cite{Nik10}.

\begin{conjecture}
\label{con}Let $k\geq2$ and let $G$ be a graph of sufficiently large order
$n.$ If $G$ has no $C_{2k+2},$ then
\[
\mu\left(  G\right)  <\mu\left(  S_{n,k}^{+}\right)  ,
\]
unless $G=S_{n,k}^{+}.$
\end{conjecture}

\bigskip

\textbf{Acknowledgement}\medskip

This work was started in 2010 when the second author was visiting the Federal
University of Rio de Janeiro, Brazil. He is grateful for the warm hospitality
of the Spectral Graph Theory Group in Rio de Janeiro. The first author was
partially supported by CNPq (the Brazilian Council for Scientific and
Technological Development).

\end{document}